\documentclass[11pt,reqno]{amsart}
\usepackage{amssymb}
\mathchardef\re="023C
\newtheorem{thm}{Theorem}

\newtheorem{defi}[thm]{Definition}

\newtheorem{lem}{Lemma}
\newtheorem{prop}{Proposition}
\newtheorem{cor}{Corollary}

\newcommand{\bcor}{\begin{cor}}
\newcommand{\ecor}{\end{cor}}
\def\be{\begin{equation}}
	\def\ee{\end{equation}}
\newcommand{\beq}{\begin{eqnarray}}
\newcommand{\beqq}{\begin{eqnarray*}}
\newcommand{\eeq}{\end{eqnarray}}
\newcommand{\eeqq}{\end{eqnarray*}}
\newcommand{\bprop}{\begin{prop}}
\newcommand{\eprop}{\end{prop}}
\newcommand{\bpf}{\begin{pf}}
\newcommand{\epf}{\end{pf}}

\newcommand{\D}{{\mathbb D}}
\newcommand{\ID}{{\mathbb D}}

\newcommand{\zb}{{\overline{z}}}
\newcommand{\dd}{\,\mathrm{d}}

\begin{document}

\bibliographystyle{plain}

\title[]{Lipschitz continuity for harmonic functions and  solutions of the \(\bar{\alpha}\)-Poisson    equation}

\author[]{Miodrag Mateljevi\' c,  Nikola Mutavd\v{z}i\' c, Adel Khalfallah}


\address{Faculty of mathematics, University of Belgrade, Studentski Trg 16, 11000 Belgrade, Serbia}
\email{\rm miodrag@matf.bg.ac.rs, nikola.math@gmail.com}


\date{}

\begin{abstract}

We study  Lipschitz continuity for  solutions of the \(\bar{\alpha}\)-Poisson    equation in planar cases. We also review some recently obtained results. As corolary we can restate results for harmonic and gradient  harmonic  functions.

\end{abstract}

\maketitle

\footnotetext[1]{\, Mathematics Subject Classification 2010 Primary 42B15, Secondary 42B30. }

\footnotetext[2]{\, Key words
and phrases: Poisson kernel, harmonic functions.}


\section{Introduction and preliminaries}
The weighted Laplacian operator $L_\rho$ is defined by $L_\rho= D_z(\rho D_{\overline{z}})$  and   $L^*_\rho= D_{\overline{z}}(\rho D_z)$.
If the weight function $\rho$  is chosen to be $\rho_\alpha (z) =(1-|z|^2)^{-\alpha}$ ($\alpha>-1$) in the unit disk $\mathbb{U}$, we call
$L_\rho$ the standard weighted Laplacian operator and write it by $L_\alpha$ for simplicity, and $\overline{L_\alpha}$ will be notation for $L^*_\rho$.

Our main result is:
\newtheorem*{main}{Theorem \ref{main}}
\begin{main}  Assume that $g \in C(\mathbb{D})$ is such that $(1-|z|^2)^\alpha g$ is bounded and $u\in V_{\mathbb{D}\rightarrow \Omega}[g]$ with the representation $u(w) = v(w) + G_{\alpha}[g](w)$. If $\alpha > 0$, and $\underline{v}$ is Lipshitz, then $u$ is also Lipshitz continuous on $\ID$.

\end{main}
 We can restate this result by means of certain solutions to $\alpha-$Poisson's equation. First we consider some basic properties of $\alpha-$harmonic mappings.
In particular, we improve Chen and Kalaj result \cite{lip3}.  Behm \cite{behm} found Green function and  solved the  Dirichlet boundary value problem of the $\alpha$-Poisson equation. Our method is based on Theorem \ref{grin-lip}  which gives estimate  of the Green potential $G_{\alpha}$ of $g$    and   the  local  $C^2$-coordinate  method  flattering the boundary  \cite{otha1-2020}.

At the begining of this paper we will introduce basic notation together with definition of so called {\it $\alpha-$Laplacian} and {\it $\alpha-$harmonic functions}. Also, definition and properties of {\it $\alpha-$Poisson's kernel} and {\it $\alpha-$Poisson's integral} are stated, as a very important technical assets used in our research. More information about this notion can be found in Olofsson's and Wittsten's paper \cite{oA}.  In the proceeding we recall definition of {\it Green function for $\alpha-$Laplacian} which is thoroughly invesigated in Behm \cite{behm}. Formulation and solution of Dirichlet boundary value problem for $\alpha-$Poisson's equation, proven in Chen and Kalaj paper \cite{ChKal}, is shown in Theorem \ref{solAlfaDir}. In paper \cite{lip1} Chen used this result to prove the boundary characterizations of a Lipschitz continuous  $\overline{\alpha}$-harmonic mappings, and proved Theorem \ref{ChenLip}.

 As an introductory result of this paper, we loose assumption on boundary value of $\overline{\alpha}$-harmonic mapping $v$, which is written in part $(d)$ of Theorem \ref{ChenLip} and attain Theorem \ref{alpha-har-lip}. Proof of this theorem is based on Hardy space technique which can be found in the first author's monography \cite{topic}, and Theorem \ref{AK-MM-thm}  proven in first author's and A. Khalfallah's paper \cite{AK-MM}.
The second improvement of Theorem \ref{ChenLip}  consider the condition on $g=-\overline{L_\alpha}u$. This result is proven in Theorem \ref{grin-lip}, and uses various estimates which we establish in Sunsection 2.3.

\subsection{$\alpha-$harmonic mappings}

Linear and semilinear equations can be tre\-a\-ted together.
We take
$a(x, y)u_x + b(x, y)u_y = c(x, y, u)$,  where a, b, c are $C^1$ functions of their arguments. The operator $a(x, y)u_x + b(x, y)u_y$
on the left hand side of this equation represents differentiation in a direction ($a,b$)
at the point $(x, y)$ in $(x, y)$-plane. Let us consider a curve, whose tangent at each
point ($x,y$) has the direction $(a, b)$. Coordinates $(x(s), y(s))$ of a point on this curve
satisfy
(1) $dx/ds = a(x, y), dy /ds = b(x, y)$
or
(2) $dy/dx = b(x, y)/a(x, y)$.

Two complex derivatives $\frac{\partial}{\partial z}=D_{z}$
and $\frac{\partial}{\partial\overline{z}}=D_{\overline{z}}$  of $u$  are written by  $$\frac{\partial}{\partial z}u=D_{z}u= \frac12(u_x -i u_y)\quad \mbox{and}\quad \frac{\partial}{\partial\overline{z}}u=D_{\overline{z}}u= \frac12(u_x +i u_y)$$  respectively, where   $z = x + iy$.\smallskip

The weighted Laplacian operator $L_\rho$ is defined by $$L_\rho= D_z(\rho D_{\overline{z}})\quad\mbox{and}\quad L^*_\rho= D_{\overline{z}}(\rho D_z).$$
If the weight function $\rho$  is chosen to be $\rho_\alpha (z) =(1-|z|^2)^{-\alpha}$   in the unit disk $\mathbb{U}$, we call
$L_\rho$   the standard weighted Laplacian operator and write it by $L_\alpha$ for simplicity,
here $\alpha$  is a real number with $\alpha > - 1$.
It is clear that
\beq
L^*_\rho u=\overline{L_\alpha} u= D_{\overline{z}}\rho   D_z u + \rho D_{\overline{z} z}u= \alpha (1-|z|^2)^{-\alpha-1} z u_z + (1-|z|^2)^{-\alpha}u_{\overline{z} z} \notag \\
L_\alpha u= \alpha (1-|z|^2)^{-\alpha-1}\overline{z} u_{\overline{z}} + (1-|z|^2)^{-\alpha}u_{\overline{z} z}\notag
\eeq

First we can see that $L^*_\alpha u =0$ iff
\be
(1-|z|^2)u_{\overline{z} z} +  \alpha  z u_z =0
\ee

It can be easily checked that $L_\alpha u=0$ iff  $L^*_\alpha \overline{u} =0.$

Set  $p=u_z$ and  $q=u_{\overline{z}}$.
Since  $\overline{u}_{\overline{z}}=\overline{u_z}$  and   $\overline{u}_z=\overline{u_{\overline{z}}} $, we find

$z u_z $  and $\overline{z} u_{\overline{z}}$    are conjugate, and also

$u_{\overline{z} z}=q_z$ and  $\overline{u}_{\overline{z} z}=\overline{q}_{\overline{z}}= \overline{q_z}$  and therefore $u_{\overline{z} z}$ and  $\overline{u}_{\overline{z} z}$ are conjugate.

If we   set   $d(z)=1-|z|^2$, then   $\rho_\alpha= d^{-\alpha}$  and  by easy computation  we find
$$\rho_z= \alpha d^{-\alpha} \overline{z},\quad \rho_{\overline{z}}= \alpha d^{-\alpha-1} z,\quad \rho_x= 2 \alpha d^{-\alpha-1}x\quad \mbox{and}\quad \rho_y= 2 \alpha d^{-\alpha-1}y.$$

Since, $2 \rho D_{z}u= \rho(u_x -i u_y)$, we find  $$4 L_\rho=D_x [\rho(u_x -i u_y)] + i D_y [\rho(u_x -i u_y)] =D_x (\rho u_x )  + D_y (\rho u_y )  + i (\rho_y u_x - \rho_x u_y).$$

Hence   $$4 L_\rho=\rho \Delta u  +  \rho_x u_x + \rho_y u_y +  i (\rho_y u_x - \rho_x u_y).$$

If $u$ is real-valued function then $L_\rho u=0$ iff  $\rho  \Delta u +\rho_x u_x + \rho_y u_y =0$  and  $yu_x -xu_y=0$, that is  $$ \Delta u  + 2 \alpha \rho_1 (xu_x +y u_y) \mbox{ and }  yu_x -xu_y=0.$$

By \cite{wikiDJ}, solutions of   $yu_x -xu_y=0$ is  $u=f(x^2+y^2)$. Since  $\rho u_z= \rho g(r) \overline{z}$, we find  $\rho g(r) r^2=z F(z)=c$  and hence  $F=0$  and  $u_z= 0$. Thus $u=c$.

If a function   $u\in C^2(\mathbb{U})$ satisfies the $\alpha$-harmonic equation
$$L_\alpha (u)= 0,$$
then we call it an $\alpha$-harmonic mapping. In the case $\alpha = 0$, $\alpha$-harmonic mappings
are just Euclidean harmonic mappings. For $L_\alpha$  notation $\Delta_\alpha$  is also used in the literature.\smallskip



\subsection{$\alpha$-Poisson's integral}
In \cite{oA}, Olofsson and Wittsten showed that if an $\alpha$-harmonic function $f$ satisfies
$$\lim_{r\to 1^{-}}f_{r}=f^{\ast} \in \mathcal{D}'(\mathbb{T}) \;\; (\alpha > -1),$$ then it has the form of a \emph{Poisson type integral}
\beq
f(z)=P_{\alpha}[f^{\ast}](z)=\frac{1}{2\pi}\int^{2\pi}_{0}P_{\alpha}(ze^{-i\theta})f^{\ast}(e^{i\theta})d\theta \nonumber
\eeq in $\mathbb{D}$,
where
$$P_{\alpha}(z)=\frac{(1-|z|^{2})^{\alpha+1}}{(1-z)(1-\overline{z})^{\alpha+1}}$$
is the {\it $\alpha$-harmonic (complex valued)  Poisson kernel} in $\ID$. In the case $\alpha=0$ we have classiacal Poisson's kernel for harmonic function and we write it as $P$ instead of $P_0$. Also, we write $P[f^*]$ instead of $P_0[f^*]$ for Piosson's integral of the function $f^*$.

\section{Lipschitz continuity for solutions of the \(\bar{\alpha}\)-Poisson    equation}

\subsection{An introductory result}\hfill

As a starting point of our investigation we used Theorem \ref{ChenLip} which can be found in Chen's paper \cite{lip1}. This theroem gives some rather strong assumption on $g=-\overline{L_\alpha} u$ ($g\in C(\ID))$ as well as for boundary values of $u$ (condition $(d)$ of Theorem \ref{ChenLip}), which are proven to be sufficient for Lipshitz continuity of $u$.

In order to formulate basic result we need to introduce some preliminary notes.
Let $V_{D\rightarrow \Omega}[g]$ denote the family of solutions of  $v: D \rightarrow \Omega $, $v\in C^2(\mathbb{D})$ of the {\it $\overline{\alpha}$-Poisson equation}
\be\label{eqHypDir}\left\{
\begin{array}{ll}
v(z)=f(z),        & \hbox{if} \,\, z\in  \mathbb{T}, \\
-(\overline{L_{\alpha}})v(z)=g, & \hbox{if}  \,\,  z\in  \mathbb{D},
\end{array}
\right.\ee
where $g \in C(\overline{\mathbb{D}})$, $f\in L^1(\mathbb{D})$ is the limit of $v(re^{i\theta} )$ as r tends to $1^-$, and $v$ is a sense-preserving diffeomorphism.\smallskip

For the case wherein the boundary function $f$ vanishes, Behm \cite{behm} solved the above Dirichlet boundary
value problem of the $\overline{\alpha}-$Poisson equation. In paper \cite{ChKal},
Chen and Kalaj combined the representation theorem given by Olofsson and Wittsten \cite{oA} with the
one given by Behm. They obtained the following theorem.

\begin{thm}[\cite{ChKal}]\label{solAlfaDir}
    Let $g\in C(\mathbb{D})$ be such that $(1-|z|^2)^{\alpha+1}g$ belongs to $L^1(\mathbb{D})$, wherein $\alpha>-1$. If $u\in C^2(\mathbb{D})$ is a solution of the equation $-\overline{L_\alpha}u=g$ satisfying the condition $u(re^{it})$ converges to a function $f\in L^1(\mathbb{T})$ as $r\to 1^-$, then for every $w\in\mathbb{D}$ we have
\begin{equation*}
    u(w)=v(w)+G_\alpha[g](w),\vspace{-2mm}
\end{equation*}
where
\begin{equation*}
    v(w){=}\frac{1}{2\pi}\int_\mathbb{T}\frac{(1{-}|w|^2)^{\alpha+1}}{(1{-}z\overline{w})(1{-}\overline{z}w)^{\alpha+1}}f(z)\dd\theta,\quad G_\alpha[g](w){=}\iint_\mathbb{D}G_\alpha(z,w)g(z)\dd x\dd y,
\end{equation*}
and the Green function $G_\alpha(z,w)$ of the adjoint Laplacian $L_\alpha$ is given by
\begin{align*}
    G_\alpha(z,w)=\frac{(1-\overline{z}w)^{\alpha}h(q(z,w))}{2\pi},\mbox{ with } z\neq w,\\
    h(r)=\frac12\int_0^{1-r^2}\frac{t^\alpha}{1-t}\dd t, q(z,w)=\left|\frac{z-w}{1-\overline{w}z}\right|.
\end{align*}
\end{thm}

In \cite{lip1}   Chen  provided  the boundary characterizations of a Lipschitz continuous  $\overline{\alpha}$-harmonic mapping as follows

Define   $\underline{f}(t)= f(e^{it})$  and
\be
S[f](w)=\frac{1}{\pi}\int_0^{2\pi} \frac{(1- |w|^2)^\alpha}{(1-z\overline{w})^\alpha}   \frac{Im (w \overline{z})}{|z-w|^2} \underline{f}(t) dt
\ee
where $z=e^{it}$.

\begin{thm}[\cite{lip1}]\label{ChenLip} Assume that $g \in C(\overline{\mathbb{D}})$ and $u\in V_{\mathbb{D}\rightarrow \Omega}[g]$ with the representation $u(w) = v(w) + G_{\alpha}[g](w)$. If $\alpha \geqslant 0$,
then the following conditions are equivalent:
\begin{itemize}
\item[$(a)$] $u$ is a (K, K')-quasiconformal mapping and $|\frac{\partial }{\partial r}v | \leq  L$ on $\mathbb{D}$, and $L$ is a constant.
\item[$(b)$] $u$ is Lipschitz continuous with the Euclidean metric.
\item[$(c)$] $v$ is Lipschitz continuous with the Euclidean metric.

\item[$(d)$]  $f$ is absolutely continuous on  $\mathbb{T}$,   $f'\in L^\infty$  and  $S[\underline{f}']$   is bounded on  $\mathbb{D}$.

\end{itemize}
\end{thm}

In order to prove the main result of this paper we will need to prove two refinement of the result above.

\subsection{Refinement of part $(d)$ in Theorem \ref{ChenLip}}\hfill

As we can see in \cite{AK-MM}, for $p\in(0,\infty]$, the {\it generalized Hardy space
$\mathcal{H}^{p}_{\mathcal{G}}(\mathbb{D})$} consists of all
measurable  functions from $\mathbb{D}$ to $\mathbb{C}$ such that
$M_{p}(r,f)$ exists for all $r\in(0,1)$, and $ \|f\|_{p}<\infty$,
where
$$M_{p}(r,f)=\left(\frac{1}{2\pi}\int_{0}^{2\pi}|f(re^{i\theta})|^{p}\,d\theta\right)^{\frac{1}{p}}
$$
and
$$\|f\|_{p}=
\begin{cases}
\displaystyle\sup \{M_{p}(r,f):\; 0<r <1\}
& \mbox{if } p\in(0,\infty),\\
\displaystyle\sup\{|f(z)|:\; z\in\mathbb{D}\} &\mbox{if } p=\infty.
\end{cases}$$

The classical {\it Hardy space $\mathcal{H}^{p}(\mathbb{D})$}  (resp. $h^{p}(\mathbb{D})$) is the set of  all  elements of $\mathcal{H}^{p}_{\mathcal{G}}(\mathbb{D})$ which are analytic (resp. harmonic) on $\D$.

\begin{defi}[\cite{topic}]
    Let $\psi\in L^1([0,2\pi])$. Then
$$H(\psi)(\varphi)=- \frac{1}{2\pi}\int\limits_{0_+}^{\pi}
\frac{\psi(\varphi+t)-\psi(\varphi-t)}{\tan t/2}\dd t,$$ denotes the
Hilbert transformations of $\psi.$
\end{defi}
Recall that   $\underline{h}(t)= h(e^{it})$. The
following property of the Hilbert transform is also sometimes taken  as
the definition:
\begin{thm}[\cite{topic}]
If $u=P[\psi]$ and $v$ is the harmonic conjugate of $u$, then
$\underline{v}(t)= H(\psi)(t)$ a.e.
\end{thm}



Note that, if $\psi$  is  $2 \pi$-periodic, absolutely continuous
on $[0,2\pi]$ (and therefore $\psi' \in L^1[0,2\pi]$), then
\begin{equation*}\label{poisson1}
\frac{\partial h}{\partial\theta} =P[\psi'].
\end{equation*}
Hence,  since   $r \frac{\partial}{\partial r} h$ is the
harmonic conjugate of $\frac{\partial}{\partial\theta} h$,  we find
\begin{align}\label{hilbert}
r \frac{\partial h}{\partial r}  &=P[H(\psi')],\quad \\
\left(\frac{\partial h}{\partial r} \right)^*(e^{i \theta})&=
H(\psi^\prime)(\theta)\quad {\rm a.e.}\,
\end{align}
Using the above outline we can derive:
\begin{thm}\label{alpha-har-lip}
Let  $h$ be harmonic  on $\mathbb{U}$ and  $h\in h^1(\mathbb{U})$. Then   $h$ is Lip  iff  $\underline{h}^\prime \in L^\infty$  and  $H(\underline{h}^\prime) \in L^\infty$.

Let  $h$  be gradient  harmonic or  $\alpha$-harmonic, $\alpha >0$, on $\mathbb{U}$ and  $h\in h^1(\mathbb{U})$. Then   $h$ is Lip  iff  $\underline{h}^\prime \in L^\infty$.
\end{thm}

The second part of the last theorem is a direct consequence of the first author's result, together with A. Khalfallah, which is stated below.
\begin{thm}[\cite{AK-MM}] \label{AK-MM-thm}
Let $\alpha \in (-1,\infty)$ with $\alpha\not=0$ and let $f=\mathcal{P}_\alpha[F]$ and $F$ is absolute continuous  such that $\dot{F}\in L^p$ with $1\leqslant p\leqslant \infty$.
\begin{enumerate}
    \item  If $\alpha>0$, then
 $\frac{\partial}{\partial \zb}f$ and $\frac{\partial}{\partial z}f$ are in  $\mathcal{H}_{\mathcal{G}}^p(\D)\subset L^p(\D)$.
 \item
If $\alpha\in (-1,0)$, then $\frac{\partial}{\partial \zb}f$ and $\frac{\partial}{\partial z}f$ are in $L^p(\D)$ for $p<- \frac{1}{\alpha}$.

\item For $\alpha \in (-1,0)$ and  $p\geqslant -\frac{1}{\alpha}$ there exists $f$ an  $\alpha$-harmonic function such that $\frac{\partial}{\partial \zb}f$   and $\frac{\partial}{\partial z}f$ $\not\in L^p(\D)$. Moreover, $\frac{\partial}{\partial \zb}f$   and $\frac{\partial}{\partial z}f$ $\not\in \mathcal{H}_{\mathcal{G}}^1(\D)$.
\end{enumerate}
  \end{thm}




\subsection{Refinement of the condition on $g=-\overline{L_\alpha}u$ in Theorem \ref{ChenLip}}\hfill

In this subsection we will prove that instead of $g\in C(\overline{\mathbb{D}})$ we can use assume that $g\in C(\mathbb{D})$ can be such that $|g(z)|\leqslant M(1-|z|^2)^{-\alpha}, z\in\mathbb{D}$, in oreder to prove Lipshitz continuity of $\overline{\alpha}$-Green integral $G_\alpha[g]$ of function $g$. This fact will play an important part in the proof of our main result.

The following two estimates can be obtained by direct investigation of the Green function $G_\alpha$, and can be found in \cite{lip1}.

\begin{align}\label{G-po-w}
    2\pi\left|\frac{\partial}{\partial w}G_\alpha(z,w)\right|&\leqslant\alpha C_\alpha |1-\overline{z}w|^{\alpha-1}\left(1{-}\left|\frac{z{-}w}{1{-}\overline{w}z}\right|^2\right)^{\alpha{+}1}\left(1{-}\log \left|\frac{z{-}w}{1{-}\overline{w}z}\right|^2\right) \notag \\
    & +\frac{(1-|z|^2)^{\alpha+1}(1-|w|^2)^{\alpha}}{2|1-z\overline{w}|^{\alpha+1}|z-w|},\vspace{-2mm}
\end{align}
\begin{align}\label{G-po-w-konj}
    2\pi\left|\frac{\partial}{\partial \overline{w}}G_\alpha(z,w)\right|\leqslant\frac{(1-|z|^2)^{\alpha+1}(1-|w|^2)^{\alpha}}{2|1-\overline{z}w|^{\alpha+1}|z-w|}.
\end{align}
In order to start with our work, we will prove the following two tehnical lemmas.
\begin{lem}\label{i2-eq}
    If $\beta >1$ then
    $$\int_0^{2\pi}\frac{\dd t}{|1-r\rho e^{i t}|^\beta}\preceq \frac{1}{|1-r\rho|^{\beta-1}}$$
    for $0<r,\rho<1$.
\end{lem}
\begin{proof}
    \begin{align*}
        \int_0^{2\pi}\frac{\dd t}{|1-r\rho e^{i t}|^\beta}&=2\int_0^{\pi}\frac{\dd t}{((1-r\rho)^2+4r\rho\sin^2\frac{t}{2})^{\beta/2}} \\ \notag
        &\leqslant \int_0^\pi\frac{\dd t}{((1-r\rho)^2+c_1t^2)^{\beta/2}}\leqslant \left|t=(1-r\rho)u\right| \\ \notag
        &\leqslant \int_0^{\pi/(1-r\rho)}\frac{(1-r\rho)\dd u}{(1-r\rho)^\beta(1+c_1u^2)^{\beta/2}} \\ \notag
        & \leqslant \frac{1}{(1-r\rho)^{\beta-1}}\int_0^{\infty}\frac{\dd u}{(1+c_1u^2)^{\beta/2}}.
    \end{align*}
since the last integral converges, we have desired result.
\end{proof}

\begin{lem}\label{i1-eq}
There exists $c_2>0$ such that
$$M_1(r)=\iint_\mathbb{D}\frac{\dd x\dd y}{|z-r|}\leqslant c_2$$
for every $0<r<1$.
\end{lem}
\begin{proof}
    Let us use the supsitution $z-r=\rho e^{i t}$, where $0\leqslant t<2\pi, 0<\rho<\rho(t)=|r-e^{it}|\leqslant r+1$. Then
    \begin{align}
        \iint_\mathbb{D}\frac{\dd x\dd y}{|z-r|}=\int_0^{2\pi}\dd t\int_0^{\rho(t)}\frac{\rho\dd\rho}{\rho}\leqslant \int_0^{2\pi}(r+1)\dd t\leqslant 4\pi.
    \end{align}
\end{proof}

Let $|w|=r$,
\begin{align*}
    I_1(w)&=\iint_\mathbb{D} \frac{(1-|z|^2)^{\alpha+1}(1-|w|^2)^{\alpha}}{2|1-\overline{z}w|^{\alpha+1}|z-w|}\dd x\dd y,  \\ 
    I_2(w)&=\iint_\mathbb{D} |1{-}\overline{z}w|^{\alpha-1}\left(1{-}\left|\frac{z{-}w}{1{-}\overline{w}z}\right|^2\right)^{\alpha{+}1}\left(1{-}\log \left|\frac{z{-}w}{1{-}\overline{w}z}\right|^2\right)\dd x\dd y. 
\end{align*}
Also, inequalities
\begin{equation}\label{trik}
|1-\overline{w}\zeta|\geqslant 1-|w|\quad \mbox{and}\quad |1-\overline{w}\zeta|\geqslant 1-|\zeta|
\end{equation}
can easily be veryfied.

The following two lemmas are crucial for the main result of this section:

\begin{lem}\label{i1-lema}
    There exists $c_3>0$ such that

    \begin{equation}
        I_1(w)\leqslant c_3 (1-|w|^2)^{\alpha}
    \end{equation}
for every $|w|<1$.
\end{lem}

\begin{proof}
    Using (\ref{trik}) we get
$$I_1(w)\preceq (1-|w|^2)^{\alpha}\iint_\mathbb{D}\frac{\dd x\dd y}{|z-w|}.$$
Since we can use coordinate change $s=\frac{w}{|w|}z$, we can use Lemma \ref{i1-eq} to get our result.
\end{proof}

Let $\zeta=\varphi_w(z)=\frac{w-z}{1-\overline{w}z}$. Then, for each $w\in\mathbb{D}, \varphi_w$ is a conformal mapping of $\mathbb{D}$ satisfying the following identities:
\begin{align}\label{ident}
     z &=\frac{w-\zeta}{1-\overline{w}\zeta}, 1-|z|^2=\frac{(1-|z|^2)(1-|w|^2)}{|1-\overline{w}\zeta|^2}, \\ \notag
     1 &-\overline{w}z=\frac{1-|w|^2}{1-\overline{w}\zeta}, \dd z=-\frac{1-|w|^2}{(1-\overline{w}\zeta)^2}\dd\zeta
\end{align}

\begin{lem}\label{i2-lema}
    There exists $c_4>0$ such that

    \begin{equation}
        I_2(w)\leqslant c_4 (1-|w|^2)^{\alpha}
    \end{equation}
for every $|w|<1$.
\end{lem}

\begin{proof}
    By using supstitution $s=\frac{w}{|w|}\zeta$, and $s=\rho e^{it}$ we get
\begin{align*}
    I_2(w)&=I_2(r)=\iint_\mathbb{D}\frac{(1-|w|^2)^{\alpha+1}(1-|\zeta|^2)^{\alpha+1}}{|1-\overline{w}\zeta|^{\alpha+3}}(1-\log |\zeta|^2)\dd\xi\dd\eta \\ \notag
    & =(1{-}|w|^2)^{\alpha} \int_0^1 (1{-}\rho^2)^{\alpha+1}(1{-}r^2)(1{-}\log \rho^2)\int_0^{2\pi}\frac{\dd t}{|1{-}r\rho e^{i t}|^{\alpha+3}}\rho\dd\rho.
\end{align*}
Using Lemma \ref{i2-eq} we get
\begin{align*}
    I_2(r)&\preceq (1-|w|^2)^{\alpha} \int_0^1\frac{(1-\rho^2)^{\alpha+1}(1-r^2)(1-\log \rho^2)}{|1-r\rho|^{\alpha+2}}\rho\dd\rho.
\end{align*}

Since $1-r\rho\geqslant 1-r$ and $1-r\rho\geqslant 1-\rho$ we have that

\begin{align*}
    I_2(r)\preceq (1-|w|^2)^{\alpha} \int_0^1\rho(1-\log \rho^2)\dd\rho \leqslant c_4 (1-|w|^2)^{\alpha}
\end{align*}

for some $c_4>0$ which does not depent of $0\leqslant r<1$.
\end{proof}

We are now ready to formulate the main result of this section, which is generalisation of Lemma 3.4 in Chen's paper \cite{lip1}. The proof of this result follows directly from Lemma \ref{i1-lema} and Lemma \ref{i2-lema}.

\begin{thm}\label{grin-lip}
    Let $g\in C(\mathbb{D})$ be such that $|g(z)|\leqslant M(1-|z|^2)^{-\alpha}, z\in\mathbb{D}$ for some $M>0$ and let $\alpha>0$ be arbitrary. Assume that $G_\alpha[g]$ is the Green potential of $g$ given by
    $$G_\alpha[g](w)=\iint_\mathbb{D} G_\alpha(z,w)g(z)\dd x\dd y.$$
    Then $\frac{\partial}{\partial w} G_\alpha[g]$ and $\frac{\partial}{\partial \overline{w}} G_\alpha[g]$ are both bounded in the unit disc $\mathbb{D}$.
\end{thm}
As a ditect consequence of Theorem \ref{grin-lip} and Theorem \ref{alpha-har-lip} we have the main result of this paper.

\begin{thm}
    \label{main}
    Assume that $g \in C(\mathbb{D})$ is such that $(1-|z|^2)^\alpha g$ is bounded and $u\in V_{\mathbb{D}\rightarrow \Omega}[g]$ with the representation $u(w) = v(w) + G_{\alpha}[g](w)$. If $\alpha > 0$, and $\underline{v}$ is Lipshitz, then $u$ is also Lipshitz continuous on $\ID$.
\end{thm}

\end{document}